\documentclass{amsart}
\usepackage{latexsym,amssymb,amsthm,amsmath,verbatim}
\usepackage{lscape,pdflscape}

\begin{document}

\newtheorem{theorem}{Theorem}[section]
\newtheorem{proposition}[theorem]{Proposition}
\newtheorem{definition}[theorem]{Definition}
\newtheorem{corollary}[theorem]{Corollary}
\newtheorem{lemma}[theorem]{Lemma}
\newtheorem{question}[theorem]{Question}

\renewcommand{\c}{\mathfrak{c}}
\newcommand{\pr}[1]{\left\langle #1 \right\rangle}
\newcommand{\mR}{\mathcal{R}}
\newcommand{\RR}{\mathbb{R}}
\newcommand{\NN}{\mathbb{N}}
\newcommand{\mA}{\mathcal{A}}
\newcommand{\mV}{\mathcal{V}}
\newcommand{\mU}{\mathcal{V}}
\newcommand{\mP}{\mathcal{P}}
\newcommand{\mD}{\mathcal{D}}
\newcommand{\mB}{\mathcal{B}}
\newcommand{\C}{\mathrm{C}}
\newcommand{\mC}{\mathcal{C}}
\newcommand{\mO}{\mathcal{O}}
\newcommand{\mM}{\mathcal{M}}
\newcommand{\mF}{\mathcal{F}}
\newcommand{\D}{\mathrm{D}}
\newcommand{\0}{\mathrm{o}}
\newcommand{\OD}{\mathrm{OD}}
\newcommand{\Do}{\D_\mathrm{o}}
\newcommand{\sone}{\mathsf{S}_1}
\newcommand{\gone}{\mathsf{G}_1}
\newcommand{\gfin}{\mathsf{G}_\mathrm{fin}}
\newcommand{\sfin}{\mathsf{S}_\mathrm{fin}}
\newcommand{\Em}{\longrightarrow}
\newcommand{\menos}{{\setminus}}
\newcommand{\w}{{\omega}}

\title{Topological games and productively countably tight
spaces}
\author{Leandro F. Aurichi \and Angelo Bella}
\thanks{This work was done during a visit of the first author to
the University of Catania, sponsored by GNSAGA}

\maketitle

\begin{abstract}
  The two main results of this work are the following: if a space
$X$ is such that player II has a winning strategy  in the game
$\gone(\Omega_x, \Omega_x)$ for every $x \in X$, then $X$ is
productively countably tight. On the other hand, if a space is
productively countably tight, then $\sone(\Omega_x, \Omega_x)$
holds for every $x \in X$. With these results, several other
results follow, using some characterizations made by Uspenskii
and Scheepers.
\end{abstract}

\section{Introduction}

Recall that a topological space $X$ is said to have {\bf
countable
tightness at a point $x \in X$} if, for every subset $A \subset
X$
such that $x \in \overline A$, there is a subset $B \subset A$
such that $x \in \overline B$ and $B$ is countable. If $X$ has
countable tightness at every point $x$, then we simply say that
$X$ has countable tightness or even that $X$ is countably tight.
The tightness does not have a good behavior  in products. It is
well known that the square of a space  of countable tightness
may
fail to have countable tightness.   An internal characterization
of those spaces $X$   such that $X \times Y$ has countable
tightness for every countably tight $Y$ was  given by
Arhangel'skii (\cite{ArhangelFrequency}). Although
Arhangel'skii's
result   works for all values of the tightness, here we will
focus
on the countable case only. Let us say that a topological space
$X$  is {\bf productively countably tight} if, for every
countably
tight space $Y$, $X \times Y$ has countable tightness. Similarly,
$X$ is productively countably tight at a point $x\in X$  provided
that,  for any space $Y$ which has countable tightness at a point
$y\in Y$, the product $X\times Y$ has countable tightness at
$\langle x,y \rangle$. In this work, we will show the relation of
this productive property with some topological games. Let us
introduce the game notation that we will use. Given two families
$\mA, \mB$, we use $\sone(\mA, \mB)$ if, for every sequence
$(A_n)_{n \in \w}$ of elements of $\mA$, one can select $a_n \in
A_n$ such that $\{a_n: n \in \w\} \in \mB$. Similarly, we use the
notation $\gone(\mA, \mB)$ for the game played between player I
and II in such a way that, for every inning $n \in \w$, player I
chooses a member $A_n \in \mA$. Then player II chooses $a_n \in
A_n$. Player II is declared the winner if, and only if, $\{a_n:
\in \w\} \in \mB$. For this matter, we will use the following
families:
\begin{itemize}
\item $\Omega$: the collection of all open $\w$-coverings for
a space (recall that $\mC$ is a $\w$-covering if, for every $F
\subset X$ finite, there is a $C \in \mC$ such that $F \subset
C$);
\item $\Omega_x$: the collection of all sets $A$ such that $x
\notin A$ and $x \in \overline A$.
\end{itemize}

In the second section we will prove that if player II has a
winning strategy  in the game $\gone(\Omega_x, \Omega_x)$ for
every $x \in X$, then $X$ is productively countably tight. On the
other hand, if $X$ is productively countably tight, then $X$ has
the property $\sone(\Omega_x, \Omega_x)$ for every $x \in X$.
Recall that $\sone(\Omega_x,\Omega_x)$ means that $X$ has
countable strong fan tightness at $x$.

In the third section, we use some translations of the properties
used here to the spaces of the form $C_p(X)$. This kind of
translation allow us to show some new results, even ones that do
not involve spaces of the form $C_p(X)$. Like, per example, if
$X$ is a Tychonoff space and player II has a winning strategy  in
the game $\gone(\Omega, \Omega)$, then the $G_\delta$
modification of $X$ is Lindel\"of.

Finally, in the last section we present  some examples  in order
to show  that the implications made in the previous sections
cannot be reversed.

\section{Productively countably tight spaces}

According to Arhangel'skii \cite{ArhangelFrequency}, a
topological space $X$ is {\bf $\aleph_0$-singular} at a point $x
\in X$  provided that there  exists  a collection $\mP$ of
centered\footnote{A family $\mF$ is centered if, for every $A, B
\in \mF$, $A \cap B \neq \emptyset$.} families of countable
subsets of $X$ such that

\begin{enumerate}

\item\label{non triviality}  for any neighborhood $O_x$ of $x$
there exists $\mB \in \mP$ and $B \in \mB$ such that $B \subset
O_x$;
\item  for any $\{\mB_n:n < \omega\} \subset \mP$  we may pick
$B_n\in \mB_n$ in such a way that $x \notin \overline
{\bigcup\{B_n: n <\omega\}}$.
\end{enumerate}

\begin{theorem}[Arhangel'skii, \cite {ArhangelFrequency}, Theorem
3.4]
Given a Tychonoff space $X$ and a point $x \in X$, $X$ is
productively countably tight at $x$ if, and only if, $X$ is not
$\aleph_0$-singular at $x$.
\end{theorem}

 With the help of this characterization, we will prove the
following:

\begin{theorem}\label{key}
  Let $X$ be a space such that player II has a winning strategy
$F$  in the game $\gone(\Omega_x, \Omega_x)$ for some $x \in X$.
Let $\mP$ be a family such that each element of $\mP$ is a
centered family of countable subsets of $X$ satisfying (\ref{non
triviality}) in the definition of $\aleph_0$-singularity.
Then, there is a family $(\mB_s)_{s \in \w^{< \w}}$ of
elements of $\mP$ such that for every choice $B_s \in \mB_s$ for
each $s \in \w^{< \w}$, $x \in \overline{\bigcup_{s \subset f}
B_s}$ for any $f \in \w^\w$.
\end{theorem}

In order to prove this Theorem, we will use the following
Lemma:

\begin{lemma}\label{there is A}
  Let $X$ be a space and let $F$ be a strategy for player II  in
the game $\gone(\Omega_x, \Omega_x)$ for some $x \in X$. Then,
for every sequence $D_0, ..., D_n \in \Omega_x$, there is an open
set $A$ such that $x \in A$ and, for every $a \in A \setminus
\{x\}$, there is a $D_a \in \Omega_x$ such that $F(D_0, ..., D_n,
D_a) = a$.
\end{lemma}

\begin{proof}
  Let $B = \{y \in X \setminus \{x\}:$ there is no $D \in
\Omega_x$ such that $F(D_0, ..., D_n, D) = y\}$. Note that $B
\notin \Omega_x$ since, otherwise, $F(D_0, ..., D_n, B) \in B$
which is a contradiction. Thus, there is an open set $A$ such
that $x \in A$ and $A \cap B = \emptyset$.
\end{proof}

\begin{proof}[Proof of Theorem \ref{key}]
  We will define by induction over $s \in \w^{<\w}$ families
$(A_s)_{s \in \w^{<\w}}$, $(D_s)_{s \in \w^{<\w}}$, $(C_s)_{s \in
\w^{<\w}}$ and $(\mB_s)_{s \in \w}$ such that, for every $s \in
\w^{< \w}$:

  \begin{enumerate}
  \item $A_s$ is an open set such that $x \in A_s$, $D_s \in
\Omega_x$, $\mB_s \in \mP$ and $C_s \in \mB_s$;
  \item $C_s \subset A_s$;
  \item $C_s = \{F(D_{(s(0))}, ..., D_s, D_{s \frown n}): n \in
\w\}$.
  \end{enumerate}

Using Lemma \ref{there is A}, let $A_\emptyset$ be an open
neighborhood of $x$ such that, for every $a \in A_\emptyset
\setminus \{x\}$, $a = F(D)$ for some $D \in \Omega_x$. Let
$\mB_\emptyset \in \mP$ be such that there is $C_\emptyset \in
\mB_\emptyset$ such that $C_\emptyset \subset A_\emptyset$.
Enumerate $C_\emptyset = \{c_n : n \in \w\}$. Let $D_{(n)}$ be
such that $c_n = F(D_{(n)})$. Now suppose all these families are
defined until $s$ and $D_{s \frown n}$ is also defined for each
$n \in \w$. Fix $n \in \w$. Using Lemma \ref{there is A}, there
is an open neighborhood $A_{s \frown n}$ of $x$ such that, for
every $a \in A_{s \frown n} \setminus \{x\}$, there is a $D \in
\Omega_x$ such that $a = F(D_{s(0)}, ..., D_{s \frown n}, D)$.
Let $\mB_{s \frown n} \in \mP$ be such that there is $C_{s \frown
n} \in \mB_{s \frown n}$ such that $C_{s \frown n} \subset A_{s
\frown n}$. Enumerate $C_{s \frown n} = \{c_k: k \in \w\}$. Let
$D_{s \frown n \frown k}$ be such that $c_k = F(D_{s(0)}, ...,
D_{s \frown n}, D_{s \frown n \frown k})$. Thus, the induction is
complete.

Now, let us show that $(\mB_s)_{s \in \w^{< \w}}$ has the desired
property. For every $s \in \w^{< \w}$, fix $B_s \in \mB_s$. Let
$f \in \w^\w$. We will show that $x \in \overline{\bigcup_{s
\subset f} B_s}$. Let $x_0 \in C_\emptyset \cap B_\emptyset$. By
our construction, there is $g(0) \in \w$ such that $x_0 =
F(D_{(g(0))})$. Then, let $x_1 \in C_{f(0)} \cap B_{f(0)}$.
Again, there is a $g(1) \in \w$ such that $x_1 = F(D_{(g(0))},
D_{(g(0), g(1))})$.  By proceeding in  this manner, since $F$ is
a winning strategy, we have  $x \in \overline{\{x_n : n \in \w\}}
\subset \overline{\bigcup_{s \subset f} B_s}$.
\end{proof}

Thus, using the characterization  of Arhangel'skii, we obtain:

\begin{corollary}\label{game implies prod}
  If $X$ is a Tychonoff space such that player II has a winning
strategy  in the game $\gone(\Omega_x, \Omega_x)$, then $X$ is
productively countably tight at $x$.
\end{corollary}

Notice that the above corollary is sharp,  as Proposition
\ref{contro} will show, \emph{i.e.}, this result is not true with
the weaker hypothesis of player I not having a winning strategy.

Two  nice classes of spaces which are productively countably
tight  are that of  bisequential spaces (see
\cite{ArhangelFrequency})  and
that of compact Hausdorff spaces of countable tightness  (see
 \cite{Malyhin}). Recall that a space  $X$ is bisequential at
$x\in X$  provided that for any ultrafilter $\mathcal {F}$
converging to
$x$ there is a decreasing  family $\{A_n:n\in \w\}\subset
\mathcal {F}$
which converges to $x$.

One could suppose that  in these spaces  player II  must have a
winning
strategy in $\gone(\Omega_x, \Omega_x)$, but it turns out  that 
this is not the case. The key point is the following:
\begin {proposition}  \label{new} \cite[Theorem
15]{Scheepers2014} 
If $X$ is a  separable $T_3$ space and player II has
a winning strategy in
$\gone(\Omega_x, \Omega_x)$, 
then $X$ is first countable at the point $x$. \end{proposition}

Now, let
$Y$ be  the ``double arrows'' space and let $Z$ be obtained by
collapsing the diagonal of $Y\times Y$ to a point $p$.  The space
$Z$ is a compact separable bisequential (and hence countably
tight) space which is not first countable
at $p$. Therefore,     from Proposition \ref{new} we see that player 
 II cannot have a winning
strategy in $\gone(\Omega_p,\Omega_p)$.

It is worth mentioning  another nice consequence of Proposition \ref{new}.

\begin{corollary} Let $X$ be a $T_3$ space and $x\in X$. If
player II has a winning strategy in $\gone(\Omega_x, \Omega_x)$,
then $X$ is strictly Fr\`echet at $x$. \end{corollary}
\begin{proof} First we show that $X$ is Fr\`echet at $x$. Let
$A\subset X$ and $x\in \overline A$. Since the tightness at $x$
is countable, there exists a countable set $B\subset A$ such that
$x\in \overline B$. Applying Proposition \ref{new} to the
separable subspace $Y=\overline B$, we see that $Y$ is first
countable at $x$ and so there is a sequence  in $B$ (and hence in
$A$) converging to $x$. Finally, by \ref{prod implies selection}
$\sone(\Omega_x, \Omega_x)$ holds and  so $X$ is strictly
Fr\`echet at $x$. \end{proof}

\begin{theorem}\label{prod implies selection}
  Let $X$ be a space. If $X$ is productively countably tight at
$p$, then  $\sone(\Omega_p, \Omega_p)$ holds, that is, $X$ has
countable strong fan tightness at $p$.
\end{theorem}

 The thesis of the Theorem will actually follow by  supposing
only  that $X \times S_\c$ has countable tightness at $\pr{p, 0}$
where

$$S_\c = \bigcup_{\alpha < \c} \{z_n^\alpha : n < \omega\} \cup
\{0\}.$$

Where all the $z_n^\alpha$'s  are distinct and isolated in $S_\c$
and, if $f \in {}^\c\omega$, then
$$V(f) = \{0\} \cup \bigcup_{\alpha < \c} \{z_n^\alpha : n \geq
f(\alpha)\}$$
is a basic open neighborhood  at $0$. Observe that the sequential
fan $S_{ \frak c}$ is a space of countable tightness.

\begin{proof}[Proof of Theorem \ref{prod implies selection}]
Let $A_n\subset X$ and  $b\in \overline{A_n}$ for
all $n\in\omega$. As $X$ has countable tightness, we may assume
that the sets $A_n$ are countable. Put $Y=
\{p\}\cup {\bigcup_{n\in\omega}A_n}$.
 Fix a  collection $\mU=\{U_\alpha :\alpha < \c\}$ of
neighborhoods of $p$ in
${X}$ whose trace on $Y$ is a local base  at $p$ in the subspace
$Y$.
Next, fix an almost disjoint family $\mR=\{R_\alpha :\alpha
<\c\}$  of
infinite subsets of $\omega$.
For every $n\in R_\alpha $ pick $x_n^{\alpha } \in A_n
\cap U_\alpha $  in such a way that $x_n^{\alpha} \ne x_m^{\alpha
}$ whenever $n\ne m$ and let $E_\alpha =\{x_n^{\alpha }:n\in
R_\alpha \}$.

Now, we work in $X\times S_\frak c$. Let
$$
     A = \bigcup_{\alpha < { \frak c}}
               \{\langle x_n^\alpha,z_n^\alpha\rangle : n \in
R_\alpha
\}.
$$
We claim that $\langle p,0\rangle \in \overline{A}$. To see this,
let $U$ and
$V(f)$ be arbitrary neighborhoods of $p$ in $X$ and $0\in
S_{ \frak c}$,
respectively. There exists $\alpha < { \frak c}$ such that
$U_\alpha\cap Y \subset  U$ and so $E_\alpha\subset  U$.  Pick $n
\in
R_\alpha  $ so large that
 $n\ge f(\alpha)$. Then
$$
     \langle x_n^\alpha,z_n^\alpha\rangle \in (U\times V(f)) \cap
A,
$$
as required.

Since $X$ is productively countably tight,    we see that
 $X\times S_{ \frak c}$ has countable tightness at $\langle
p,0\rangle$. Therefore,    there is a
countable subset $F$ of ${ \frak c}$ such that if
$$
     B = \bigcup_{\alpha \in F}
               \{\langle x_n^\alpha,z_n^\alpha\rangle : n \in
R_\alpha
\},
$$
then $\langle p,0\rangle \in \overline{B}$.   We claim that

\begin{equation}\label{eq:*}
\text{for every $U$ neighborhood of $p$, $|U\cap E_\alpha
|=\aleph_0$ for some $\alpha \in F$.}
\end{equation}

Striving for a contradiction, assume that there is a neighborhood
$U$ of $p$ such that for every $\alpha\in F$, $|U\cap E_\alpha| <
\aleph_0$. Let $n_\alpha \in \omega$ be such that, if $n \in
R_\alpha \setminus   n_\alpha$,  then $x_n^\alpha \not \in U$.
Define $f\in  {}^{ \frak c}\omega $ as follows: $f(\alpha) =
n_\alpha$ if $\alpha \in F$ and $f(\beta) = 0$ otherwise.
Pick $\alpha \in F$ and $n \in R_\alpha  $ such that
$\pr{x_n^\alpha,z_n^\alpha} \in U\times V(f)$. Then
$x_n^\alpha\in U\cap E_\alpha$ and so $n < n_\alpha$.
However, $z_n^\alpha\in V(f)$ which implies $n \geq f(\alpha) =
n_\alpha$. This is a contradiction.

 Enumerate $F=\{\alpha _n:n<\omega\}$.  Then let $S_{\alpha
_0}=R_{\alpha _0}$ and $S_{\alpha _k}=R_{\alpha _k}\setminus
(R_{\alpha _0}\cup\cdots \cup R_{\alpha _{k-1}})$ for $k>0$.
The sets $S_{\alpha _k}$ are pairwise disjoint  and $S_{\alpha
_k}$     differs from $R_{\alpha _k}$  only in finitely many
points.
Let $n\in \omega$. If $n\in {S}_{\alpha _m}$ for some (single!)
$m$,
then put $a_n = x_n^{\alpha _m}$. Otherwise pick $a_n \in A_n$
arbitrarily. Let $U$ be an open neighborhood of $p$. By
$(\ref{eq:*})$, there is an $\alpha \in F$ such that $|U \cap
S_\alpha| = \aleph_0$. Thus, there is an $x_n^{\alpha_k} \in U$.
Note that $a_n = x_n^{\alpha_k}$ so we proved that $p \in
\overline{\{a_n: n \in \w\}}$.
\end{proof}

\begin{corollary}\label{prod implies s}
  If a space $X$ is productively countably tight, then
$\sone(\Omega_x, \Omega_x)$ holds for each $x\in X$.
\end{corollary}
Scheepers has  pointed out (see p. 250-251 in \cite
{Scheepers1997} or Theorem 11 in \cite{Scheepers2014})   that
$\sone(\Omega_x,
\Omega_x)$ can be strictly weaker than player I not having  a
winning strategy in $\gone(\Omega_x, \Omega_x)$. So, the previous
results immediately suggest:
\begin{question} Let $X$ be a space and assume that $X$ is
productively countably tight at a point $x\in X$. Is it then true
that player I does not have a winning strategy in
$\gone(\Omega_x, \Omega_x)$? \end{question}

\section{Some applications}

Using some translations of properties from $X$ to $C_p(X)$ and
vice-versa, we get some applications for the theorems from the
previous section. We begin this section stating the translations
that we will use. Recall that, given a space $X$, we call the
{\bf $G_\delta$ modification} of $X$ (denoted by $X_\delta$) the
space where all the $G_\delta$'s of $X$ are declared open. Also,
we will use the notation $\0$ for the function constantly equal
to $0$.

\begin{theorem}[Scheepers \cite{Scheepers2014}]\label{Scheepers}
  Let $X$ be a Tychonoff space. Player II  has a winning strategy
in $\gone(\Omega,\Omega)$ played on $X$  if, and only if,  player
II has a winning strategy in
$\gone(\Omega_{\0},\Omega_{\0})$ played on $C_p(X)$.
\end{theorem}

\begin{theorem}[Uspenskii \cite{Uspenskii1982}]\label{Uspenskii}
For any Tychonoff space $X$, $C_p(X)$ is productively countably
tight if, and only if, $X_\delta$ is Lindel\"of.
\end{theorem}

\begin{corollary}
  Let $X$ be a Tychonoff space. If $C_p(X)$ is productively
countably tight, then $X$ is productively Lindel\"of.
\end{corollary}

\begin{proof}
  Just note that in \cite{Aurichi2013}, it is proved that, if
$X_\delta$ is Lindel\"of, then $X$ is productively Lindel\"of.
\end{proof}

\begin{proposition} \label{prop}
  Let $X$ be a Tychonoff space. If player II has a winning
strategy in $\gone(\mO, \mO)$, then player II has a winning
strategy in $\gone(\Omega_{\0},\Omega_{\0})$ played on $C_p(X)$
and $C_p(X)$ is productively countably tight.
\end{proposition}

\begin{proof}
   In \cite{Aurichi2013} it is shown that, if player II has a
winning strategy  in
$\gone(\mO, \mO)$ played on $X$, then $X_\delta$ is Lindel\"of.
Therefore, by Theorem \ref{Uspenskii}, $C_p(X)$ is productively
countably tight.

  Also,  if player II has a winning strategy  in $\gone(\mO,
\mO)$,
then player II has a winning strategy  in $\gone(\Omega, \Omega)$
(\cite{Gerlits1982}, Theorem 1). Therefore, by Theorem
\ref{Scheepers}, player II has a
winning strategy in $\gone(\Omega_{\0},\Omega_{\0})$ played on
$C_p(X)$.
\end{proof}

As mentioned in the proof of \ref{prop},
the first author and Dias proved in \cite{Aurichi2013} that, if
player II has a winning strategy  in the game $\gone(\mO, \mO)$,
then $X_\delta$ is Lindel\"of. Here we can do a little better, at
least for Tychonoff spaces:

\begin{proposition}
  Let $X$ be a Tychonoff space. If player II has a winning
strategy  in the game $\gone(\Omega, \Omega)$, then $X_\delta$ is
Lindel\"of.
\end{proposition}

\begin{proof}
  It follows directly from Theorem \ref{Scheepers}, Corollary
\ref{game implies prod} and Theorem \ref{Uspenskii}.
\end{proof}

\section{Some examples}

We begin this section by showing that the implication in
Corollary
\ref{game implies prod} cannot be reversed, even for spaces of
the form $C_p(X)$:

\begin{proposition}\label{Telagarsky}
  There is a space of the form $C_p(X)$ that is productively
countably tight but such that player II does not have a winning
strategy in the game $\gone(\Omega_{\0}, \Omega_{\0})$.
\end{proposition}

\begin{proof}
  An example $X$ due to Telgarsky, mentioned by the first author
and Dias in
 \cite{Aurichi2013}, Example 3.4, provides a Lindel\"of P-space
in which player II does not have a winning strategy in
$\gfin(\mO, \mO)$. Therefore, player II does not have a winning
strategy even  in $\gone(\Omega, \Omega)$.  Note that, since
$X_\delta=X $ is Lindel\"of, $C_p(X)$
is productively countably tight (by Theorem \ref{Uspenskii}).
But, since player II does not have a winning strategy  in the
game $\gone(\Omega, \Omega)$ played on $X$, then player II does
not have a winning strategy  in the game $\gone(\Omega_{\0},
\Omega_{\0})$ played on $C_p(X)$ (by Theorem \ref{Scheepers}).
\end{proof}

In the proof of Theorem \ref{prod implies selection},
we only used the fact that $X \times S_\c$ has countable
tightness to show that $X$ satisfies $\sone(\Omega_x, \Omega_x)$.
A natural question  arises on whether the other direction also
works. We will show that this is not the case, at least
consistently.

First, recall that a point $p\in \omega^*$ is a {\bf selective
ultrafilter}
if for any partition $\{A_n:n<\omega\}$ of $\omega$ such that
$A_n\notin p$ for each $n$, there exists $P\in p$ such that
$|P\cap A_n|=1$ for each $n$. The existence of such $p$ is
independent from ZFC (see \cite{Kunen1976}).

\begin{proposition}\label{selective}
  Let $p\in \omega^*$ be a selective ultrafilter  and consider
the space $\{p\}\cup \omega$. Then, $\sone(\Omega_p,\Omega_p)$
holds, but $(\{p\}\cup
\omega)\times S_\frak c$ is not countably tight at $\pr{p,0}$.
\end{proposition}

\begin{proof}
We begin by showing that $\sone(\Omega_p, \Omega_p)$ holds in
$\{p\}\cup \omega$. Let  $A_n\subset
\omega$ with $p\in \overline {A_n}$ for every $n<\omega$. Since
$A_n\in p$ for each $n$, even $B_n=A_0 \cap \cdots \cap A_n\in
p$. If $\bigcap \{B_n:n<\omega\}\in p$, then we are done. In the
other case,  since the family $\{\omega\setminus
B_0\}\cup\{B_n\setminus B_{n+1}:n<\omega\}\cup
\{\bigcap\{B_n:n<\omega\}\}$ is a partition of $\omega$ whose
elements are  not in $p$. Since $p$ is selective, there is a $Q
\in p$
such that $Q=\{q_n:n<\omega\}$ and $q_n\in B_n\setminus B_{n+1}
\subset  A_n$ for
each $n$. Therefore, $S^\omega_1(\Omega_p,\Omega_p)$ holds.

Now, let $p=\{P_\alpha :\alpha <\frak c\}$ and for each $\alpha $
put
$S_\alpha =\omega\setminus P_\alpha =\{x_n^\alpha :n<\omega\}$.
Let $A=\{\pr{x_n^\alpha ,z_n^\alpha}: n < \omega, \alpha <\frak
c\}$. We claim that $\langle p,0\rangle \in \overline A$.
 To see this, let $P\in p$ and $f\in {}^\frak c\omega$.
 Split $P$
into two infinite sets, say $P'$ and $P''$. Without loss of
generality, assume that $P'\in p$. Then $\omega\setminus
P'=S_\beta$ for some $\beta<\frak c$. It  follows that
$(\{p\}\cup P)\times V(f) \cap A$ is infinite and the claim is
proved. Let $B$ be a countable subset of $A$.  Then there exists
a countable  $F\subset  \frak c$ such that $B\subset  \{\langle
x_n^\alpha ,z_n^\alpha \rangle: n<\omega, \alpha \in F\}$. Since
$p$ is a  selective ultrafilter,   there exists $Q\in p$ such
that  $|Q\cap
S_\alpha |<\aleph_0$ for each $\alpha \in F$.  Then we may find a
function $g\in {}^\frak c\omega$ such that $(\{p\}\cup Q)\times
V_g\cap B=\emptyset$.  This shows that $(\{p\}\cup \omega)\times
S_\frak c$ does not have countable tightness at $\langle
p,0\rangle$.
\end{proof}

Thus, in particular, the property of $X$ being productively
countably tight at $x \in X$ is strictly in between
$\sone(\Omega_x,
\Omega_x)$ and player II having a winning strategy  in the game
$\gone(\Omega_x, \Omega_x)$.

In \cite{Scheepers1997}, Theorem 13B ($1 \Rightarrow 2$),
Scheepers proved the following:

\begin{proposition}\label{tight cov}
  Let $X$ be a space of countable tightness. For every $x \in X$,
if $\chi(x, X) < cov(\mM)$, then player I does not have a winning
strategy in the game $\gone(\Omega_x, \Omega_x)$.
\end{proposition}

We can show that the hypotheses of Proposition \ref{tight cov}
are not strong enough to guarantee that player II has a winning
strategy. Of course,  this is not the case of
$cov(\mM)=\aleph_1$, because  player II has an obvious winning
strategy  at any point of first countability.

\begin{proposition}
   [$cov(\mM)>\aleph_1$]\  There is a space $X$ of countable
tightness such
that $\chi(X) < cov(\mM)$ and player II does not have a winning
strategy in the game $\gone(\Omega_x, \Omega_x)$ for every $x \in
X$.
\end{proposition}

\begin{proof}
Assume $cov(\mM) > \aleph_1$  and take a subset $Y \subset \RR$
of
cardinality $\aleph_1$. Then let $X = C_p(Y)$. Note that $X$
has countable tightness since every finite power of $Y$ is
Lindel\"of and $\chi(X) = \aleph_1 < cov(\mM)$. Note that
$Y_\delta$ is not Lindel\"of
and so by Theorem \ref{Uspenskii}, $X$ is not productively
countably tight. Finally, thanks to Theorem \ref{prod implies
selection}, we see that player II does not have a winning
strategy in $\gone(\Omega_x, \Omega_x)$ for every $x \in X$.
\end{proof}

The same proof of the above  result yields also the following:
\begin{proposition}[$cov(\mM)>\aleph_1$] \label{contro}  \  There
is a Tychonoff space $X$  which is not  productively countably
tight and player I does not have a winning strategy in
$\gone(\Omega_x, \Omega_x)$ for each $x\in X$. \end{proposition}
We finish   with some  results, obtained  with the help of
Pixley-Roy hyperspaces. Recall that  the Pixley-Roy hyperspace
$PR(X)$ over  a space $X$ consists of the set of all non-empty
finite subsets of $X$ with the topology generated by   the base
$\{[A,U] :A\in PR(X), U$ open in $X\}$, where $[A,U]=\{B\in
PR(X) :A\subset B\subset U\}$.

A family $\mathcal {S}$ of non-empty subsets of a space $X$ is a
$\pi$-network at a point $x\in X$ if  every neighbourhood of $x$
contains an element of $\mathcal {S}$.

For a given space $X$ and a point  $x\in X$,
let us denote by $\pi\mathcal {N}_x$ the collection of all
$\pi$-networks
at
$x$   consisting of finite sets.

With only minor modifications in the proof of Theorem 1.1 in
\cite{bella} (just replacing $\gfin$ with $\gone$), we may prove
the following:
\begin{theorem} \label{bella}  Let $X$ be a space, $A=\{x_1,
\ldots, x_k\}\in PR(X)$ and $p=(x_1,\ldots,x_k)\in X^k$. Then,
player  II has a winning strategy in $\gone(\Omega_A, \Omega_A)$,
played  on $PR(X)$, if and only if player II has a winning
strategy  in $\gone(\pi\mathcal {N}_p,\pi\mathcal {N}_p)$, played
on
$X^k$.
\end{theorem}

 In \cite{bellasakai}, the second author and Sakai established
when a
Pixley-Roy  hyperspace is productively countably tight.

A space
$X$ is supertight at a point $x\in X$ provided that  for any
$\pi$-network $\mathcal {P}$   at $x$ consisting of countable
sets
there
exists a countable subfamily $\mathcal {Q}\subset \mathcal {P}$
which
is still a
$\pi$-network at $x$.
\begin{theorem} [Bella-Sakai, \cite{bellasakai}, Theorem 2.6]
\label{bellasakai}  A Pixley-Roy hyperspace
$PR(X)$ is productively countably tight if and only if $X$ is
supertight at every point.
\end{theorem}
 From the equivalences in Theorems \ref{bella} and
\ref{bellasakai} and  Corollary \ref{game implies prod}, we
immediately get the following:
\begin{theorem} \label{final}  Let $X$ be a Tychonoff space  and
$x\in X$. If
player II has a winning strategy in the game
$\gone(\pi\mathcal {N}_x,\pi\mathcal {N}_x)$, then $X$ is
supertight
at $x$.
\end{theorem}
\begin{proof} By Theorem \ref {bella}, player II has a winning
strategy in $\gone(\Omega_{\{x\}}, \Omega_{\{x\}})$ played on
$PR(X)$.  Then, by Corollary \ref{game implies prod}, $PR(X)$ is
productively countably tight at $\{x\}$. Finally, a closer
inspection  at  the proof of Theorem 2.6  in \cite{bellasakai}
shows that even a local version of it holds.  This suffices to
conclude that $X$ is supertight at $x$.
\end{proof}
This theorem, together with the fact that  any supertight space
is productively countably tight (see \cite{bellasakai}, Corollary
2.3),   may suggest:
\begin{question} \label{q3}  Let $X$ be a space, $x\in X$  and
assume that player II has a winning strategy in $\gone(\Omega_x,
\Omega_x)$.  Is it true that $X$ is supertight at $x$?
\end{question}


\def\cprime{$'$}

\end{document}